\documentclass[12pt]{amsart}
\usepackage{amsmath}
\usepackage{amssymb}
\usepackage{amsthm}

\input xy
\xyoption{matrix}
\xyoption{arrow}
\xyoption{2cell}
\UseTwocells

\theoremstyle{plain}
\newtheorem{theorem}{Theorem}%[section]
\newtheorem{corollary}[theorem]{Corollary}

\newtheorem{lemma}[theorem]{Lemma}

\theoremstyle{definition}
\newtheorem{definition}[theorem]{Definition}

\newtheorem{observation}[theorem]{Observation}

\newcommand{\br}[1]{\langle#1\rangle}
\newcommand{\C}{{\mathcal{C}}}

\newcommand{\Cat}{{\mathbf{Cat}_*}}

\def\d{\displaystyle}
\newcommand{\D}{\mathcal{D}}

\newcommand{\id}{{\mathop{\textnormal{id}}\nolimits}}
\long\def\ignore#1\endignore{}

\newcommand{\Mult}{{\mathbf{Mult}}}
\newcommand{\Multstar}{\Mult_*}
\newcommand{\N}{{\mathbb{N}}}

\newcommand{\Ob}{{\mathop{\textnormal{Ob}}}}

\newcommand{\op}{{\textnormal{op}}}
\newcommand{\Perm}{\mathbf{Perm}}

\newcommand{\sm}{\wedge}
\newcommand{\Strict}{{\mathbf{Strict}}}

\def\u#1{\underline{#1}}

\begin{document}

\title{Multiplicativity in Mandell's Inverse $K$-Theory}

\author{A.\ D.\ Elmendorf}

\address{Department of Mathematics\\
Purdue University Northwest\\
 Hammond, IN 46323}
 \thanks{The author was supported in part by a Simons Foundation
 Collaboration Grant for Mathematicians}

\email{adelmend@pnw.edu}

\date{\today}

\begin{abstract}
We show that Mandell's inverse $K$-theory functor from $\Gamma$-categories to permutative categories preserves multiplicative structure.  
This is a first step towards an equivariant generalization that would be inverse to the construction of Bohmann
and Osorno.
\end{abstract}

\maketitle

% main text
\section{Introduction}
Segal showed that symmetric monoidal categories give rise to spectra in \cite{Seg}, and May gave a simpler construction in the case of permutative categories in \cite{May}.  Both constructions deserve to be called algebraic $K$-theory constructions.
Mandell and the author showed in \cite{EM1} that a modification of May's construction actually preserves multiplicative structure,
and Bohmann and Osorno \cite{BO} used this multiplicativity to construct an equivariant version, making crucial use of the work
of Guillou and May \cite{GM} characterizing equivariant spectra as presheaves of spectra over a spectral version of the Burnside
category.  

All the spectra arising from the Segal-May construction are connective, and Thomason showed in \cite{Thom1} that all connective
spectra arise in this fashion.  His construction was quite obscure, however, and Mandell gave a much more comprehensible 
one in \cite{Man}.  The aim of this paper is to show that the main step in Mandell's construction preserves multiplicative structure, and therefore can serve as a first step towards giving an inverse construction to the equivariant generalization of Bohmann and Osorno.  The conjecture is that all connective equivariant spectra, that is, those all of whose fixed-point spectra are connective, arise from the Bohmann-Osorno construction.  

The first part of Mandell's construction is simply to use Thomason's equivalence between spaces and categories from \cite{Thom2} to construct a $\Gamma$-category from a $\Gamma$-space, by using levelwise double subdivision and categorification.  The substantial portion of Mandell's construction is the passage from $\Gamma$-categories to permutative categories, and it is this part of his construction that we show preserves multiplicative structure.
The proof relies on a factorization of Mandell's construction as a composite of three functors, all of which preserve multiplicative structure.  The first starts with a $\Gamma$-category and produces a multifunctor with source a multicategory derived from the natural numbers, and with target a
multicategory of categories.  The second step is a wreath product construction that starts from such a multifunctor and 
produces a single multicategory.  The third step is the left adjoint to the forgetful functor from permutative categories to multicategories; it was somewhat surprising to the author to find that this also preserves multiplicative structure, but the proof is fairly simple once one thinks to look for it.

It is a pleasure to acknowledge stimulating conversations about this material with Bert Guillou, Peter Bonventre, and especially Ang\'elica
Osorno.  It seems like a good idea to acknowledge Anna Marie Bohmann on general inspirational principles.  None of them
are responsible for the errors and omissions that may occur in this paper.

\section{Outline and statement of results}
We begin with a $\Gamma$-category $X$, and wish to end with a permutative category using a construction that preserves
multiplicative structure, as captured by multicategory structure.  To set terminology and notation, let $\Gamma^\op$ be the category with objects the based sets
$\u n=\{0,1,\dots,n\}$  with basepoint 0 for $n\ge0$, and morphisms the based functions.  
Let $\Cat$ be the category of small based categories, that is, small categories with a select base object.
Then a $\Gamma$-category
is a functor $X:\Gamma^\op\to\Cat$ for which $X(0)=*$, a category with one object and one morphism.

Next, let $\N$ temporarily denote the permutative category whose objects are the unbased sets 
$[n]=\{1,2,\dots,n\}$ for $n\ge0$, morphisms all functions, and monoidal product given by $[m]\oplus[n]:=[m+n]$,
where we use the canonical bijection $[m]\amalg[n]\cong[m+n]$ to make this into a bifunctor.  Being a permutative
category, so is its opposite category, and our permanent use of $\N$ will be as the underlying multicategory of the
permutative category given by this opposite category.  Explicitly, then, an $r$-morphism in $\N$ from $([n_1],\dots,[n_r])$
to $[m]$ consists of an ordinary function $[m]\to[n_1]\amalg\cdots\amalg[n_r]$.  It is important to note that $\N$ is actually
a based multicategory: there is a canonical multifunctor 
from the terminal multicategory $*$ with one object and one morphism of each arity to our multicategory $\N$,
picking out the single object $[0]=\emptyset$ and making
it into a commutative monoid in the multicategorical sense: this just means that the 
basepoint-defining multifunctor $*\to\N$ lands on
the object $[0]$.  The terminology about commutative monoids arises from the fact that $*$ parametrizes commutative
monoids in any symmetric monoidal category.

We now define a multicategory structure on $\Cat$, or more precisely its opposite category.  
Note that a based category is the same thing as a category with a specified 
functor $*\to\C$ from any terminal category with one object and one morphism.
Since $\Cat$ is a symmetric
monoidal category using its categorical product, which is the cartesian product of categories, so is its opposite category.
We will use $\Cat^\op$ to denote ``the'' underlying multicategory of this opposite category, where by ``the'' underlying
multicategory we mean any choice of underlying multicategory: all of them are canonically isomorphic.  Explicitly,
an $r$-morphism in $\Cat^\op$ from $(\D_1,\dots,\D_r)$ to $\C$ consists of a based functor
\[
\C\to\D_1\times\cdots\times\D_r.
\]
We note that this multicategory is itself based, with base object any one-point, one-morphism category.

Since both $\N$ and $\Cat^\op$ are based multicategories, we can look at 
\[
\Mult_*(\N,\Cat^\op),
\] 
the collection of
based multifunctors from $\N$ to $\Cat^\op$.  The first step in our construction takes a $\Gamma$-category $X$
and produces a multifunctor $AX\in\Mult_*(\N,\Cat^\op)$; we use the notation $A$ to follow Mandell's notation to
some extent, and the construction is given in Section \ref{G2M}.

\begin{theorem}\label{step1}
There are multicategory structures on $\Gamma$-$\Cat$ and $\Mult_*(\N,\Cat^\op)$ for which the construction
\[
A:\Gamma\text{-}\Cat\to\Mult_*(\N,\Cat^\op)
\]
is a multifunctor.
\end{theorem}

In fact, the multicategory structure on $\Gamma$-$\Cat$ arises from the
symmetric monoidal structure given by the Day convolution induced by the smash product
of based categories and the smash product $\Gamma^\op\times\Gamma^\op\to\Gamma^\op$.  However, the multicategory
structure on $\Mult_*(\N,\Cat^\op)$ is \emph{not} the one given by the enrichment of $\Mult_*$ over itself; we explain
in Section \ref{G2M}.

\ignore
Here is the construction of $AX$ given a $\Gamma$-category $X$.  We need to first assign a based category $AX[n]$ to any
object $[n]\in\N$, and we just use the category $X(n)$.  Next, given a morphism $f:([n_1],\dots,[n_r])\to[m]$ in $\N$,
we need an associated morphism $(AX[n_1],\dots,AX[n_r])\to AX[m]$ in $\Cat^\op$.  But the morphism $f$ consists
just of a function
\[
f:[m]\to[n_1]\amalg\cdots\amalg[n_r],
\]
to which we can attach a disjoint basepoint on each side, considered as new elements 0, and obtain a map of based sets
\[
f_+:\u m\to\u n_1\vee\cdots\vee\u n_r.
\]
For any index $i$ with $1\le i\le r$, we can then collapse all the wedge summands except $\u n_i$ to the basepoint,
producing a map $f_i:\u m\to\u n_i$ in $\Gamma^\op$.  Since $X$ is a $\Gamma$-category, this induces a functor 
$X(f_i):X(m)\to X(n_i)$ which we use as the $i$'th coordinate map to the product of the $X(n_i)$'s, giving us
a functor
\[
AX(f):X(m)\to X(n_1)\times\cdots\times X(n_r);
\]
this functor is based since all the components are induced by maps in a $\Gamma$-category, which must be based functors.
\endignore

The next step is a wreath product construction that produces a based multicategory from a 
(based) multifunctor
into $\Cat^\op$.  The construction will be given in Section \ref{MWP}.
\ignore
Suppose given a based multicategory $M$ and a based multifunctor $F:M\to\Cat^\op$.  We define a based multicategory
$M\wr F$ as follows.  The objects of $M\wr F$ are given as
\[
\Ob(M\wr F):=\coprod_{a\in M}\Ob(Fa),
\]
with the base object given by the unique object of the terminal category to which the base object of $M$ gets mapped.
Given a source string $(x_1,\dots,x_r)$ with $x_i\in Fa_i$ and a target object $y\in Fb$, a morphism 
$(x_1,\dots,x_r)\to y$ in $M\wr F$ consists of 
\begin{enumerate}
\item
an $r$-morphism $f:(a_1,\dots,a_r)\to b$ in $M$, which induces an $r$-morphism $Ff:(Fa_1,\dots,Fa_r)\to Fb$
in $\Cat^\op$, in other words, a functor $Ff:Fb\to Fa_1\times\cdots\times Fa_r$,
\item
an $r$-tuple of morphisms $\psi_i\in Fa_i$ assembling to 
\[
\prod_{i=1}^r\psi_i:(x_1,\dots,x_r)\to(Ff)(y)
\]
as a morphism in $\prod_{i=1}^r Fa_i$.
\end{enumerate}
We now assert
\endignore
We will show that
\begin{theorem}\label{step2}
The wreath product construction gives a multifunctor
\[
\mathord{\textnormal{Wr}}:\Mult_*(\N,\Cat^\op)\to\Mult_*.
\]
\end{theorem}
Here the multicategory structure on the target $\Mult_*$ is the one underlying the symmetric monoidal structure from
\cite{EM2}.

At this point in the construction we forget about the based structure of the objects of $\Multstar$, that is,
we apply the forgetful functor $\Multstar\to\Mult$.  This is a multifunctor since based multilinear maps of based 
multicategories are in particular multilinear maps of their underlying unbased multicategories.  Now
the third and final step in our factorization of Mandell's construction is the left adjoint to the forgetful functor from 
permutative categories to multicategories,
which we denote by $F$ and describe in Section \ref{FPC}.
\ignore
Given a permutative category $\C$, its underlying multicategory
has the same objects, and an $r$-morphism $(a_1,\dots,a_r)\to b$ consists of a morphism in $\C$
\[
f:a_1\oplus\cdots\oplus a_r\to b.
\]
If $r=0$, we consider an empty sum to be given by the identity object of the permutative category.

The left adjoint to this construction is as follows.  Given a multicategory $M$, we construct a permutative category
$FM$ by first specifying its objects to be
\[
\Ob(FM):=\coprod_{n=0}^\infty\left(\Ob M\right)^n,
\]
so the objects of $FM$ consists of lists of objects of $M$, including an empty list, which gives the identity object.
Given a source string $\br{x_i}_{i=1}^r=(x_1,\dots,x_r)$ and a target string $\br{y_j}_{j=1}^s=(y_1,\dots,y_s)$, we define
a morphism $\br{x_i}_{i=1}^r\to\br{y_j}_{j=1}^s$ to consist of a function $\phi:\{1,\dots,r\}\to\{1,\dots,s\}$ and, for each $j$ with $1\le j\le s$,
a morphism $\psi_j:\br{x_i}_{\phi(i)=j}\to y_j$ in $M$, where $\br{x_i}_{\phi(i)=j}$ is the tuple of entries in $\br{x_i}_{i=1}^r$ whose indices get mapped to $j$.  If there are no such $i$, then $\br{x_i}_{\phi(i)=j}$ is the empty list, and $\psi_j$ is a 0-morphism in $M$.  The permutative structure is
given by concatenation of lists.  
\endignore
Let $\Perm$ denote the multicategory of permutative categories.  
%\endignore
Our third multiplicativity theorem is then

\begin{theorem}\label{step3}
The functor $F:\Mult\to\Perm$ extends to a multifunctor; i.e., it preserves multiplicative structure.
\end{theorem}

The following observation is our factorization of Mandell's construction.

\begin{observation}
Given a $\Gamma$-category $X$, Mandell's construction $PX$ in \cite{Man} of a permutative category coincides with
the construction $F(\N\wr AX)$.
\end{observation}

Since all three steps in this construction preserve multiplicative structure, we may conclude

\begin{corollary}
Mandell's construction $P$ defines a multifunctor from $\Gamma$-$\Cat$ to $\Perm$ inverse to the $K$-theory construction.
\end{corollary}

The fact that $P$ is inverse to the $K$-theory construction is one of Mandell's results from \cite{Man}.

\section{Multifunctors from symmetric monoidal categories}
This section is devoted to a lemma that we will exploit in the proofs of the above theorems.  
We refer to the unit $e$ in a symmetric monoidal category as
\emph{strong} if the natural map $e\otimes x\to x$ is an isomorphism.

\begin{lemma}\label{unbased}
Let $(\C,\otimes,e)$ be a symmetric monoidal category with strong unit $e$, and let $M$ be a multicategory.  Let $U\C$ be
``the'' underlying multicategory of $\C$ (there are many choices, but they're all canonically isomorphic.)  Then
a multifunctor $F:U\C\to M$ determines and is determined by:
\begin{enumerate}
\item
a functor on underlying categories,
\item
a 2-morphism $\lambda:(Fx,Fy)\to F(x\otimes y)$ for each pair of objects $x$ and $y$, and
\item
a 0-morphism $\eta:()\to Fe$,
\end{enumerate}
all subject to the four diagrams listed below.
\end{lemma}

The diagrams that we require $\lambda$ 
and $\eta$ 
to satisfy are as follows.  First, we require $\lambda$ to be
coherently associative, in the sense that
\begin{equation}\label{Ass}
\xymatrix{
&(F(x\otimes y),Fz)\ar[r]^-{\lambda}
&F((x\otimes y)\otimes z)\ar@{<->}[dd]^-{F\alpha}_-{\cong}
\\(Fx,Fy,Fz)\ar[ur]^-{(\lambda, 1)}\ar[dr]_-{(1,\lambda)}
\\&(Fx,F(y\otimes z))\ar[r]_-{\lambda}
&F(x\otimes(y\otimes z))
}
\end{equation}
commutes, where $\alpha$ is the associativity isomorphism in $\C$.  We also require it to be consistent with the 
interchange isomorphism, in the sense that
\begin{equation}\label{Comm}
\xymatrix{
(Fx,Fy)\ar[r]^-{\tau}\ar[d]_-{\lambda}
&(Fy,Fx)\ar[d]^-{\lambda}
\\F(x\otimes y)\ar[r]_-{F\tau}&F(y\otimes x)
}
\end{equation}
commutes.  
In addition, we require a naturality condition for $\lambda$: given morphisms $f_1:x_1\to y_1$ and $f_2:x_2\to y_2$
in $\C$, we require the following diagram to commute:
\begin{equation}\label{Nat}
\xymatrix@C+15pt{
(Fx_1,Fx_2)\ar[r]^-{(Ff_1,Ff_2)}\ar[d]_-{\lambda}
&(Fy_1,Fy_2)\ar[d]^-{\lambda}
\\F(x_1\otimes x_2)\ar[r]_-{F(f_1\otimes f_2)}
&F(y_1\otimes y_2).
}
\end{equation}
Finally, we require $\lambda$  to be coherent with $\eta$
in the sense that the following diagram commutes:
\begin{equation}\label{Unit}
\xymatrix{
Fx\ar[r]^-{(\eta,1)}\ar[drr]_-{=}
&(Fe,Fx)\ar[r]^-{\lambda}
&F(e\otimes x)\ar[d]^-{\cong}
\\&&Fx.
}
\end{equation}
(The corresponding diagram for $x\otimes e$ now follows from diagram \ref{Comm}.)

\begin{proof}
Suppose first given a multifunctor $F:U\C\to M$.  Then certainly $F$ descends to a functor on the underlying categories.  
We also have a canonical and natural 2-morphism 
\[
(x,y)\to x\otimes y
\]
in $U\C$ given by the identity morphism on $x\otimes y$, and applying $F$ gives us our 2-morphism
\[
\lambda:(Fx,Fy)\to F(x\otimes y)
\]
in $M$.  

Since 0-morphisms in $U\C$ are given by morphisms from $e_\C$, we have a canonical 0-morphism 
$()\to e_\C$ given
by the identity on $e_\C$, and applying $F$ gives us
\[
\eta:()\to F(e_\C).
\]

The required diagrams for $\lambda$
 and $\eta$ 
 now follow from the properties of a symmetric monoidal category.

Now suppose given just a functor $F:U\C\to M$ on underlying categories, together with a 2-morphism $\lambda$ 
and 0-morphism $\eta$
subject to the given diagrams.  We wish to extend $F$ to a multifunctor.

Let $(x_1,\dots,x_r)$ be an $r$-tuple of objects in $\C$.  
We have by induction an $r$-morphism in $M$ 
\[
\lambda_r:(Fx_1,\dots,Fx_r)\to F(x_1\otimes\cdots\otimes x_r).
\]
Explicitly, we take $\lambda_0=\eta$, and 
\[
\lambda_{r+1}:=\lambda\circ(\lambda_r,\id).
\]  
Note in particular that $\lambda_1=\id$, by diagram \ref{Unit}.
Now given an arbitrary $r$-morphism $f:(x_1,\dots,x_r)\to y$ in $U\C$, there is a specified 1-morphism
$f:x_1\otimes\cdots\otimes x_r\to y$ in $\C$, and we agree to apply the ordinary functor $F$ to
this morphism and define the $r$-morphism $Ff$ in $M$ by the composite
\[
\xymatrix{
(Fx_1,\dots,Fx_r)\ar[r]^-{\lambda_r}
&F(x_1\otimes\cdots\otimes x_r)\ar[r]^-{Ff}
&Fy
}
\]
Note that if $f$ happens to be a 0-morphism, we have an empty tensor product which we agree to be $e_\C$,
and the definition still makes sense. We claim that these assignments
make $F$ into a multifunctor.  

We must first show that $F$ preserves the composition in $U\C$.  So suppose that for $1\le i\le s$, we have a $t_i$-morphism
$g_i:\br{w_{ij}}_{j=1}^{t_i}\to x_i$, and an $s$-morphism $f:\br{x_i}_{i=1}^s\to y$.  We need to show that
\[
F(f\circ\br{g_i}_{i=1}^s)=Ff\circ\br{F(g_i)}_{i=1}^s.
\]
Let $t=\sum_{i=1}^s t_i$; both sides of the proposed equality are $t$-morphisms.  We have the following diagram, which we claim commutes, where $\odot$ denotes concatenation of lists:
\[
\xymatrix{
\br{\br{Fw_{ij}}_{j=1}^{t_i}}_{i=1}^s\ar[d]_-{\br{\lambda_{t_i}}}\ar[r]^-{\odot}
&\bigodot_{i=1}^s\br{Fw_{ij}}_{j=1}^{t_i}\ar[d]^-{\lambda_t}
\\\br{F(\otimes_{j=1}^{t_i}w_{ij})}_{i=1}^s\ar[r]^-{\lambda_s}\ar[r]^-{\lambda_s}\ar[d]_-{\br{Fg_i}}
&F\left(\otimes_{i=1}^s\otimes_{j=1}^{t_i}w_{ij}\right)\ar[d]^-{F(\otimes_{i=1}^sg_i)}
\\\br{Fx_i}_{i=1}^s\ar[r]_-{\lambda_s}
&F\left(\otimes_{i=1}^sx_i\right)\ar[r]^-{Ff}
&Fy.
}
\]
Tracing counterclockwise gives the right hand side, and clockwise gives the left hand side.  The top square commutes
up to an explicit associativity isomorphism
by induction using the associativity of $\lambda$, or the unit diagram in case any of the $g_i$ are 0-morphisms, and the bottom square by induction using the naturality of $\lambda$.
The extended functor therefore preserves composition in the sense of multicategories.

We must also show that $F$ preserves the permutation actions on sets of morphisms.  Since permutations are generated
by transpositions, it suffices to show preservation of transpositions.  Suppose given a 2-morphism $f:(x,y)\to z$ in $U\C$,
which amounts to a 1-morphism $f:x\otimes y\to z$, abusively denoted by the same letter.  Then $f\tau$ is given by
the composite
\[
\xymatrix{
y\otimes x\ar[r]^-{\tau}&x\otimes y\ar[r]^-{f}&z.
}
\]
Now the commutative diagram
\[
\xymatrix{
(Fy,Fx)\ar[r]^-{\tau}\ar[d]_-{\lambda}
&(Fx,Fy)\ar[d]^-{\lambda}
\\F(y\otimes x)\ar[r]_-{F\tau}
&F(x\otimes y)\ar[r]_-{Ff}&z
}
\]
shows that $F(f\tau)=(Ff)\tau$.  We have therefore shown that $F$ extends to a multifunctor. 
It is straightforward to check that the two constructions are inverse to each other.
\end{proof}

\ignore
The based version of the lemma is as follows.

\begin{lemma}\label{based}
Let $(\C,\otimes,e)$ be a symmetric monoidal category with strong unit $e$, and let $M$ be a
based multicategory.  Let $U\C$ be
``the'' underlying based multicategory of $\C$ (there are many choices, but they're all canonically isomorphic.)  Then
a based multifunctor $F:U\C\to M$ determines and is determined by:
\begin{enumerate}
\item
a based functor on underlying categories, and
\item
a 2-morphism $\lambda:(Fx,Fy)\to F(x\otimes y)$ for each pair of objects $x$ and $y$,
\end{enumerate}
subject to the following conditions.  
\end{lemma}
First, we declare $\eta:()\to F(e_\C)$ to be the canonical 0-morphism in $M$
arising from $F(e_\C)$ being the base object of $M$.  Next, we require all four diagrams from Lemma \ref{unbased}
to commute.  And we also require the following diagram to commute, where $\mu$ is the product map
on $F(e_\C)$ arising from it being the base object of $M$:
\begin{equation}\label{monoid}
\xymatrix{
(F(e_\C),F(e_\C))\ar[r]^-{\lambda}\ar[dr]_-{\mu}
&F(e_\C\otimes e_\C)\ar[d]^-{F(\mu)}
\\&F(e_\C).
}
\end{equation}

\begin{proof}
If $F$ is a based multifunctor, then we obtain $\lambda$ as before, and the fact that $F(e_\C)$ is the base object
of $M$ forces all the diagrams to commute; diagram \ref{monoid} follows since $F$ preserves base objects.

Now if $F$ is a based functor together with $\lambda$ satisfying the five conditions, then we obtain a multifunctor
as before, and it must be based because of diagram \ref{monoid}.  This completes the proof.
\end{proof}
\endignore

\section{From $\Gamma$-categories to multifunctors}\label{G2M}
This section is devoted to proving Theorem \ref{step1}, that the construction $A:\Gamma$-$\Cat\to\Mult_*(\N,\Cat^\op)$,
given as follows,
extends to a multifunctor.  

Here is the construction of $AX$ given a $\Gamma$-category $X$.  We need to first assign a based category $AX[n]$ to any
object $[n]\in\N$, and we just use the category $X(n)$.  Next, given a morphism $f:([n_1],\dots,[n_r])\to[m]$ in $\N$,
we need an associated morphism $(AX[n_1],\dots,AX[n_r])\to AX[m]$ in $\Cat^\op$.  But the morphism $f$ consists
just of a function
\[
f:[m]\to[n_1]\amalg\cdots\amalg[n_r],
\]
to which we can attach a disjoint basepoint on each side, considered as new elements 0, and obtain a map of based sets
\[
f_+:\u m\to\u n_1\vee\cdots\vee\u n_r.
\]
For any index $i$ with $1\le i\le r$, we can then collapse all the wedge summands except $\u n_i$ to the basepoint,
producing a map $f_i:\u m\to\u n_i$ in $\Gamma^\op$.  Since $X$ is a $\Gamma$-category, this induces a functor 
$X(f_i):X(m)\to X(n_i)$ which we use as the $i$'th coordinate map to the product of the $X(n_i)$'s, giving us
a functor
\[
AX(f):X(m)\to X(n_1)\times\cdots\times X(n_r);
\]
this functor is based since all the components are induced by maps in a $\Gamma$-category, which must be based functors.
This concludes the description of the functor $A$; we now turn to showing it extends to a multifunctor structure, for which
we need multicategory structures on both source and target.

For the multicategory structure on $\Gamma$-$\Cat$, we exploit the Day convolution construction.  There is a permutative category structure on $\Gamma^\op$
in which we use lexicographic order to identify $\u m\sm\u n$ with $\u{m\cdot n}$.  Now given two $\Gamma$-categories
$X$ and $Y$, noting that for all $n$ both $X(n)$ and $Y(n)$ are based categories, we form the smash product
$X\sm Y$ as the left Kan extension
\[
\xymatrix{
\Gamma^\op\times\Gamma^\op\ar[r]^-{X\times Y}\ar[d]_-{\sm}
\drtwocell<\omit>
&\Cat\times\Cat\ar[d]^-{\sm}
\\\Gamma^\op\ar[r]_-{X\sm Y}
&\Cat.
}
\]
This gives us a symmetric monoidal structure on $\Gamma$-$\Cat$;  
%\ignore
the unit object is the $\Gamma$-category $B$ given by
\[
B(n)=
\begin{cases}
S^0&\text{ if }n=1,
\\{*}&\text{ if }n\ne1.
\end{cases}
\]
Here $S^0$ is a two-object discrete category with one object the base; it is the unit for the smash product of based categories.
%\endignore
We therefore have an underlying multicategory structure on $\Gamma$-$\Cat$.

To give the multicategory structure on $\Mult_*(\N,\Cat^\op)$, we need extra structure on both $\N$ and $\Cat^\op$, which
we call \emph{ring} structure.

\begin{definition}
Let $M$ be a based multicategory.  A \emph{ring} structure on $M$ consists of a based bilinear map of multicategories
\[
\lambda:(M,M)\to M
\]
together with a unit multifunctor $\eta:u\to M$, where $u$ is the unit for the smash product of based multicategories 
(see \cite{EM2} for further details; such a multifunctor amounts simply to a choice of object of $M$.) 
These are then subject to the same associativity and symmetry conditions one
imposes on a symmetric monoidal category.
\ignore
the same diagrams one expects for a permutative category, 
namely the following:
\endignore
\end{definition}

\ignore
We note that a ring structure on a multicategory $M$ is the same thing as an action of the categorical Barratt-Eccles
operad on $M$ using the multicategory structure on $\Mult_*$.
\endignore

The ring structures we have in mind are given by the following two lemmas.

\begin{lemma}
The multicategory $\N$ has a ring structure given by 
\[
\lambda([m],[n]):=[m\cdot n],
\]
 using lexicographic order to 
identify $[m\cdot n]$ with $[m]\times[n]$ in order to extend to morphisms.  The unit object is $[1]=\{1\}$.
\end{lemma}

The proof is straightforward and left to the reader: bilinearity of $\lambda$ is an aspect of the distributive law.

The next lemma gives the ring structure on $\Cat^\op$.

\begin{lemma}
The multicategory $\Cat^\op$ has a ring structure given by
\[
\lambda(\C,\D):=\C\sm\D.
\]
Given an $r$-morphism $f:(\C_1,\dots,\C_r)\to\C$, i.e., a functor $f:\C\to\C_1\times\cdots\times\C_r$, the induced
$r$-morphism $\lambda(f,\D):(\C_1\sm\D,\dots,\C_r\sm\D)\to\C\sm\D$ is given by the functor
\[
\xymatrix{
\C\sm\D\ar[r]%^-{f\times\Delta}
&(\C_1\sm\D)\times(\C_2\sm\D)\times\cdots\times(\C_r\sm\D)
}
\]
whose $i$th coordinate map is given by $f_i\sm\D:\C\sm\D\to\C_i\sm\D$,
and similarly in the variable $\D$.  The unit category is the unit $S^0$ for the smash product of based categories:
it is discrete with two objects, one of which is the base object.
\end{lemma}

\begin{proof}
The main issue is to show that the multifunctorialities in $\C$ and $\D$ interact correctly.  For this, suppose
given an $r$-map
$f:\C\to\C_1\times\cdots\times\C_r$ and an $s$-map $g:\D\to\D_1\times\cdots\times\D_s$. Then we have
\[
\xymatrix@C+15pt{
&\prod_{i=1}^r(\C_i\sm\D)\ar[r]^-{\prod\lambda(\C_i,g)}
&\prod_{i=1}^r\prod_{j=1}^s(\C_i\sm\D_j)\ar@{<->}[dd]^-{\cong}
\\\C\sm\D\ar[ur]^-{\lambda(f,\D)}\ar[dr]_-{\lambda(\C,g)}
\\&\prod_{j=1}^s(\C\sm\D_j)\ar[r]_-{\prod\lambda(f,\D_j)}
&\prod_{j=1}^s\prod_{i=1}^r(\C_i\sm\D_j),
}
\]
which commutes since the projection to the factor $\C_i\sm\D_j$ is always just $f_i\sm g_j$.
This
establishes the coherence necessary to have a bilinear map of multicategories.  The other verifications are
straightforward.
\end{proof}

Given a ring structure $\lambda:(M,M)\to M$, we can iterate and obtain a canonical $r$-linear map
$\lambda_r:(M,\dots,M)\to M$, where there are $r$ copies of $M$ in the source.  We think of the unit
map as being $\lambda_0$.  We can now define the multicategory structure on $\Mult_*(\N,\Cat^\op)$ of
interest to us; as mentioned above, this is \emph{not} the multicategory structure given by the enrichment
of $\Mult_*$ over itself given by its symmetric monoidal structure, which makes no use of the ring structures.
Instead, this is a multicategory structure that twists by means of the ring structures.  

\begin{definition}
Let $F_1,\dots,F_r$, and $G$ be elements of $\Mult_*(\N,\Cat^\op)$.  We define an $r$-map
$(F_1,\dots,F_r)\to G$ to be a based $r$-linear transformation $\phi$ as in the following diagram:
\[
\xymatrix@C+20pt{
(\N,\dots,\N)\ar[r]^-{(F_1,\dots,F_r)}\ar[d]_-{\lambda_r}
\drtwocell<\omit>{\phi}
&(\Cat^\op,\dots,\Cat^\op)\ar[d]^-{\lambda_r}
\\\N\ar[r]_-{G}
&\Cat^\op.
}
\]
\end{definition}

Unpacking this a bit, for each $r$-tuple $([m_1],\dots,[m_r])$ of objects of $\N$, we require a based functor
\[
\phi_{m_1,\dots,m_r}:F_1m_1\sm\cdots\sm F_rm_r\to G(m_1\cdots m_r)
\]
that is multifunctorial in each variable $m_i$ separately, basepoint preserving, and based bilinear 
(in the sense of \cite{EM2}, Definition 2.8)
in each pair of variables.  To illustrate
in just the case of a bilinear transformation $(F_1,F_2)\to G$, we require based functors
\[
\phi_{m,n}:F_1m\sm F_2n\to G(mn)
\]
for all $m, n\in\N$, and given an $r$-morphism $f:(m_1,\dots,m_r)\to m'$ in $\N$, that is, a function
\[
f:[m']\to[m_1]\amalg\cdots\amalg[m_r],
\]
we require the following diagram to commute:
\[
\xymatrix@C+10pt{
F_1m'\sm F_2n\ar[r]^-{\phi_{m',n}}\ar[d]_-{\lambda_2(F_1f,F_2n)}
&G(m'n)\ar[d]^-{G(f\times[n])}
\\\prod_{i=1}^r(F_1m_i\sm F_2n)\ar[r]_-{\prod\phi_{m_i,n}}
&\prod_{i=1}^rG(m_in),
}
\]
along with an analogous diagram in the variable $n$.  Given in addition an $s$-morphism
\[
g:[n']\to[n_1]\amalg\cdots\amalg[n_s],
\]
we 
also require the bilinearity diagram
\[
\xymatrix{
&F_1m'\sm F_2n'\ar[dl]_-{\lambda_2(F_1f,F_2n')\phantom{mm}}\ar[dr]^-{\phantom{m}\lambda_2(F_1m',F_2g)}
\\\prod_{i=1}^r(F_1m_i\sm F_2n')\ar[d]_-{\prod\phi_{m_i,n'}}
&&\prod_{j=1}^s(F_1m'\sm F_2n_j)\ar[d]^-{\prod\phi_{m',n_j}}
\\\prod_{i=1}^rG(m_in')\ar[d]_-{\prod G(m_i\cdot g)}
&&\prod_{j=1}^sG(m'n_j)\ar[d]^-{\prod G(f\cdot n_j)}
\\\prod_{i=1}^r\prod_{j=1}^sG(m_in_j)\ar[rr]_-{\cong}
&&\prod_{j=1}^s\prod_{i=1}^rG(m_in_j)
}
\]
to commute.  These diagrams are then modified as appropriate for larger numbers of variables,
but with no significant differences.

We can now begin the proof of Theorem \ref{step1}.  Since $\Gamma$-$\Cat$ forms a symmetric monoidal category,
Lemma \ref{unbased} tells us that we need 
to check that $A$ gives us a functor,
specify a natural 2-morphism $\lambda:(AX,AY)\to A(X\sm Y)$ in
$\Mult_*(\N,\Cat^\op)$ for $\Gamma$-categories $X$ and $Y$, specify
a 0-morphism $\eta:()\to A(B)$, and show that these
choices satisfy the four diagrams listed after Lemma \ref{unbased}.

For functoriality, given a morphism $q:X\to Y$ of $\Gamma$-categories, we get a 1-morphism $Aq:AX\to AY$ of multifunctors in 
$\Mult_*(\N,\Cat^\op)$ as follows.  For each $[m]\in\N$, we have $q[m]: X(m)\to Y(m)$ as the component of the natural
map $q$ at the object $[m]$ of $\Gamma^\op$; this gives us the required morphisms $AX[m]=X(m)\to Y(m)=AY[m]$.  To see
that these maps give us a multinatural transformation, suppose given an $r$-morphism $f:[m']\to[m_1]\amalg\cdots
\amalg[m_r]$ in $\N$.  Then we get the induced maps $f_i:\u m'\to\u m_i$ for each $i$ with $1\le i\le r$, and naturality
of $q$ now tells us that 
\[
\xymatrix{
X(m')\ar[r]^-{q(m')}\ar[d]_-{X(f_i)}
&Y(m')\ar[d]^-{Y(f_i)}
\\X(m_i)\ar[r]_-{q(m_i)}
&Y(m_i)
}
\]
commutes.  Since each $X(f_i)$ and $Y(f_i)$ are the coordinate maps for the induced map in $\Cat^\op$, we find that
\[
\xymatrix{
X(m')\ar[r]^-{q(m')}\ar[d]_-{AX(f)}
&Y(m')\ar[d]^-{AY(f)}
\\\prod_{i=1}^rX(m_i)\ar[r]_-{\prod q(m_i)}
&\prod_{i=1}^rY(m_i)
}
\]
commutes, showing that $Aq$ is in fact multinatural.  It is easy to see that composition of 1-morphisms is preserved.

For the construction of $\lambda:(AX,AY)\to A(X\sm Y)$, we use the natural transformations giving $X\sm Y$ as
a left Kan extension: just pasting on the functor $\N\to\Gamma^\op$ that attaches the disjoint basepoint 0, we get
the necessary transformations from the pasting diagram
\[
\xymatrix{
\N\times\N\ar[r]\ar[d]_-{\lambda_2}
&\Gamma^\op\times\Gamma^\op\ar[r]^-{X\times Y}\ar[d]^-{\sm}\drtwocell<\omit>
&\Cat^\op\times\Cat^\op\ar[d]^-{\sm}
\\\N\ar[r]
&\Gamma^\op\ar[r]_-{X\sm Y}
&\Cat^\op.
}
\]
The first three coherence diagrams now follow from the universal property of the left Kan extension defining $X\sm Y$.

To define the 0-morphism $\eta:()\to A(B)$, we note that the smash product of an empty list of based categories
must be the unit $S^0$ for the smash product, and the product of an empty list of integers is 1.  We therefore 
just need to specify a map of based categories from $S^0$ to $B(1)$, but since $B(1)=S^0$, we can just use 
the identity functor.  The fourth coherence diagram now follows.
This concludes the proof of Theorem \ref{step1}.

\section{The Multicategorical Wreath Product}\label{MWP}

This section is devoted to proving Theorem \ref{step2}, which states that the multicategorical wreath product provides
a multifunctor from $\Mult_*(\N,\Cat^\op)$ to $\Mult_*$.  
We begin with the description of the construction.

Suppose given a based multicategory $M$ and a based multifunctor $F:M\to\Cat^\op$.  We define a based multicategory
$M\wr F$ as follows.  The objects of $M\wr F$ are given as
\[
\Ob(M\wr F):=\coprod_{a\in M}\Ob(Fa),
\]
with the base object given by the unique object of the terminal category to which the base object of $M$ gets mapped.
Given a source string $(x_1,\dots,x_r)$ with $x_i\in Fa_i$ and a target object $y\in Fb$, a morphism 
$(x_1,\dots,x_r)\to y$ in $M\wr F$ consists of 
\begin{enumerate}
\item
an $r$-morphism $f:(a_1,\dots,a_r)\to b$ in $M$, which induces an $r$-morphism $Ff:(Fa_1,\dots,Fa_r)\to Fb$
in $\Cat^\op$, in other words, a functor $Ff:Fb\to Fa_1\times\cdots\times Fa_r$,
\item
an $r$-tuple of morphisms $\psi_i\in Fa_i$ assembling to 
\[
\prod_{i=1}^r\psi_i:(x_1,\dots,x_r)\to(Ff)(y)
\]
as a morphism in $\prod_{i=1}^r Fa_i$.
\end{enumerate}
Composition is now straightforward to construct.
This completes the description of the multicategorical wreath product we wish to use (there are other variants.)

Now given an $r$-morphism $(F_1,\dots,F_r)\to G$
in $\Mult_*(\N,\Cat^\op)$, we need to produce an $r$-morphism $(\N\wr F_1,\dots,\N\wr F_r)\to\N\wr G$ in $\Mult_*$.  
We start with just the case $r=2$, so suppose given a 2-morphism $\phi:(F_1,F_2)\to G$, and we wish to produce a based bilinear
map of based multicategories $\N\wr\phi:(\N\wr F_1,\N\wr F_2)\to\N\wr G$.

On objects, given a pair of objects $(x,m)$ and $(y,n)$ of $\N\wr F_1$ and $\N\wr F_2$
respectively, so $x\in F_1[m]$ and $y\in F_2[n]$, the 2-morphism $\phi$ provides us with an object $\phi_{m,n}(x,y)\in G[m\cdot n]$,
so we can define the map on objects by
\[
(\N\wr\phi)((x,m),(y,n)):=(\phi_{m,n}(x,y),m\cdot n)\in\N\wr G.
\]
We need to see that this assignment can be made based multifunctorial in each variable, and based bilinear.  

Suppose we have an $r$-morphism $((x_1,m_1),\dots,(x_r,m_r))\to(x',m')$ in $\N\wr F_1$, so this unpacks as
an $r$-morphism
\[
f:[m']\to[m_1]\amalg\cdots\amalg[m_r]
\]
in $\N$, which induces
\[
F_1f:F_1m'\to F_1m_1\times\cdots\times F_1m_r,
\]
an $r$-morphism in $\Cat^\op$; we then also require an $r$-tuple of morphisms $\br{\psi_i}_{i=1}^r$ giving a morphism
in $\prod_{i=1}^rF_1m_i$ from $(x_1,\dots,x_r)$ to $(F_1f)(x')$.  We wish to construct from these data an
$r$-morphism $\br{(\phi_{m_i,n}(x_i,y),m_i\cdot n)}\to(\phi_{m',n}(x',y),m'\cdot n)$ in $\N\wr G$.

We start with the $r$-morphism in $\N$ given by
\[
f\times[n]: [m'\cdot n]\to[m_1\cdot n]\amalg\cdots\amalg[m_r\cdot n]
\]
obtained from the lexicographic identification $[m_i]\times[n]\cong[m_i\cdot n]$.  
This induces the functor $G(f\times[n]):G(m'n)\to\prod_{i=1}^rG(m_in)$, and we also desire an $r$-tuple of morphisms
from $\br{\phi_{m_i,n}(x_i,y)}$ to $G(f\times[n])(\phi_{m',n}(x',y))$ in $\prod_{i=1}^rG(m_i,n)$.
We know that the square
\[
\xymatrix@C+10pt{
F_1m'\sm F_2n\ar[r]^-{\phi_{m',n}}\ar[d]_-{\lambda_2(F_1f,F_2n)}
&G(m'n)\ar[d]^-{G(f\times[n])}
\\\prod_{i=1}^r(F_1m_i\sm F_2n)\ar[r]_-{\prod\phi_{m_i,n}}
&\prod_{i=1}^rG(m_in),
}
\]
commutes, so
\[
G(f\times[n])(\phi_{m',n}(x',y))=\bigg(\prod_{i=1}^r\phi_{m_i,n}\bigg)((F_1f)(x'),y)
=\br{\phi_{m_i,n}\left((F_1)f(x')_i,y\right)}_{i=1}^r,
\]
where we write $(F_1)(x')_i$ for the component of $(F_1f)(x')$ in $F_1m_i$.
But now we can exploit the bifunctoriality of all the $\phi_{m_i,n}$'s to insert the $\br{\psi_i}$'s into the first slot,
giving us an $r$-tuple of maps
\[
\br{\phi_{m_i,n}(\psi_i,y)}_{i=1}^r:\br{\phi_{m_i,n}(x_i,y)}_{i=1}^r\to
\br{\phi_{m_i,n}\left((F_1)f(x')_i,y\right)}_{i=1}^r
%\bigg(\prod_{i=1}^r\phi_{m_i,n}\bigg)((F_1f)(x'),y)
\]
which assemble to a single morphism 
in $\prod_{i=1}^rG(m_in)$.  We have therefore produced the desired $r$-morphism in $\N\wr G$. We proceed
similarly given an $s$-morphism in $\N\wr F_2$.  If there are a larger number of $F$'s, the appropriate modifications
to the construction are straightforward.  Preservation of composition requires several large diagrams that the reader
is encouraged to construct for herself.

For the bilinearity diagram, suppose we have an $r$-morphism
\[
(f,\br{\psi_i}_{i=1}^r):\br{(x_i,m_i)}_{i=1}^r\to(x',m')
\]
in $\N\wr F_1$, and an $s$-morphism
\[
(g,\br{\xi_j}_{j=1}^s):\br{(y_j,n_j)}_{j=1}^s\to(y',n')
\]
in $\N\wr F_2$.  We wish the following diagram to commute in $\N\wr G$:
\[
\xymatrix@C+15pt{
\br{\br{(\N\wr\phi)((x_i,m_i),(y_j,n_j))}_{i=1}^r}_{j=1}^s
\ar[r]^-{(f,\br{\psi_i})_*}
\ar@{<->}[d]_-{\cong}
&\br{(N\wr\phi)((x',m'),(y_j,n_j))}_{j=1}^s
\ar[dd]^-{(g,\br{\xi_j})_*}
\\\br{\br{(\N\wr\phi)((x_i,m_i),(y_j,n_j))}_{j=1}^s}_{i=1}^r
\ar[d]_-{(g,\br{\xi_j})_*}
\\\br{(\N\wr\phi)((x_i,m_i),(y',n'))}_{i=1}^r
\ar[r]_-{(f,\br{\psi_i})_*}
&(\N\wr\phi)((x',m'),(y',n')).
}
\]
Tracing clockwise, we have the $r$-tuple of morphisms 
\[
\br{\psi_i}_{i=1}^r:\br{x_i}_{i=1}^r\to(F_1f)(x'),
\]
or again writing $(F_1f)(x')_i$ for the component of $(F_1f)(x')$ in $F_1m_i$, we have 
\[
\psi_i:x_i\to(F_1f)(x')_i
\]
as a morphism in $F_1m_i$ for $1\le i\le r$.  Exploiting the bifunctoriality of $\phi$ as required, we
obtain an $r$-tuple of morphisms for all $j$ as follows:
\[
\phi_{m_i,n_j}(\psi_i,y_j):\phi_{m_i,n_j}(x_i,y_j)\to\phi_{m_i,n_j}((F_1f)(x'),y_j).
\]
Since we have
\[
G(f\times[n_j])(\phi_{m',n_j}(x',y_j))
=\br{\phi_{m_i,n_j}((F_1f)(x'),y_j)}_{i=1}^r,
\]
these assemble to give us the required morphism defining an $r$-map in $\N\wr G$ for each $1\le j\le s$.

We now wish to compose this $s$-tuple of morphisms in $\N\wr G$ with the single $s$-morphism
induced by $(g,\br{\xi_j}_{j=1}^s)$.  Explicitly, we have 
\[
g:[n']\to\coprod_{j=1}^s[n_j]
\]
which induces
\[
F_2g:F_2n'\to\prod_{j=1}^sF_2n_j,
\]
and we also have
\[
\xi_j:y_j\to(F_2g)(y')_j
\]
as a morphism in $F_2n_j$ for all $j$.  In the same manner as before, we have
\[
G([m']\times g)(\phi_{m',n'}(x',y'))=\br{\phi_{m',n_j}(x',(F_2g)(y')_j)}_{j=1}^s,
\]
and so bifunctoriality of $\phi_{m',n_j}$ gives us maps
\[
\phi_{m',n_j}(x',\xi_j):\phi_{m',n_j}(x',y_j)\to\phi_{m',n_j}(x',(F_2g)(y')_j)
\]
for all $j$, giving us the required $s$-morphism in $\N\wr G$.

Now composing these data, we first have the composite in $\N$ given by
\[
\xymatrix@C+25pt{
[m'n']\ar[r]^-{[m']\times g}
&\d\coprod_{j=1}^s[m'n_j]\ar[r]^-{\coprod(f\times[n_j])}
&\d\coprod_{j=1}^s\coprod_{i=1}^r[m_in_j]
}
\]
inducing the functor composite
\[
\xymatrix@C+25pt{
G(m'n')\ar[r]^-{G([m']\times g)}
&\d\prod_{j=1}^sG(m'n_j)\ar[r]^-{\prod(G(f\times[n_j]))}
&\d\prod_{j=1}^s\prod_{i=1}^rG(m_in_j).
}
\]
But because $G$ is a multifunctor, and therefore commutes with permutations in 
the ``source,'' (actually the target of the functors, since we're working in $\Cat^\op$),
this is all actually equivalent up to reordering the factors in the products to
\[
G(f\times g):G(m'n')\to\prod_{(i,j)\in\u r\times\u s}G(m_in_j).
\]

Next, we need to compose the morphisms
\[
G(f\times[n_j])(\phi_{m',n_j}(x',\xi_j))
=\br{\phi_{m_i,n_j}((F_1f)(x'),\xi_j)}_{i=1}^r,
\]
whose components we can identify and rewrite as
\begin{gather*}
G(f\times[n_j])_i(\phi_{m',n_j}(x',\xi_j))
:G(f\times[n_j])_i(\phi_{m',n_j}(x',y_j))
\\\to G(f\times[n_j])_i(\phi_{m',n_j}(x',(F_2g)(y')_j))
\\=\phi_{m_i,n_j}((F_1f)(x')_i,(F_2g)(y')_j),
\end{gather*}
with the morphisms
\[
\phi_{m_i,n_j}(\psi_i,y_j):\phi_{m_i,n_j}(x_i,y_j)\to G(f\times[n_j])_i(\phi_{m',n_j}(x',y_j)).
\]
But as before, we have
\[
G(f\times[n_j])_i(\phi_{m',n_j}(x',\xi_j))
=\phi_{m_i,n_j}((F_1f)(x')_i,\xi_j),
\]
so the composites we need to form end up being
\[
\phi_{m_i,n_j}((F_1f)(x')_i,\xi_j)\circ\phi_{m_i,n_j}(\psi_i,y_j)
=\phi_{m_i,n_j}(\psi_i,\xi_j),
\]
which is independent of the priority order of the indices.  We therefore get the same result on composing
in the other order, establishing the bilinearity in $\N\wr G$.

\section{The Free Permutative Category on a Multicategory}\label{FPC}

In this section we prove Theorem \ref{step3}: the left adjoint to the forgetful functor from permutative categories to 
multicategories is actually a multifunctor. 
\ignore
First, however, we observe that the construction of Mandell's inverse $K$-theory functor doesn't involve the actual
left adjoint to forgetting from permutative categories to based multicategories, but rather the composite of the forgetful
functor from based to unbased multicategories, followed by the left adjoint to the forgetful functor from permutative 
categories to unbased multicategories, whose construction was given in Section 1.  Since based multilinear maps
are in particular multilinear when forgetting about the basepoint structure, the forgetful functor from based 
multicategories to unbased ones is a multifunctor.  It remains to show that the left adjoint from $\Mult$ to $\Strict$
is also a multifunctor.
\endignore

We begin by describing the forgetful functor and its left adjoint $F$.
Given a permutative category $\C$, its underlying multicategory
has the same objects, and an $r$-morphism $(a_1,\dots,a_r)\to b$ consists of a morphism in $\C$
\[
f:a_1\oplus\cdots\oplus a_r\to b.
\]
If $r=0$, we consider an empty sum to be given by the identity object of the permutative category.
Composition is given by taking sums of sums, and composing within $\C$.  The $\Sigma_n$-actions are induced
from the transposition isomorphism in the permutative structure.

The left adjoint to this construction is as follows.  Given a multicategory $M$, we construct a permutative category
$FM$ by first specifying its objects to be
\[
\Ob(FM):=\coprod_{n=0}^\infty\left(\Ob M\right)^n,
\]
so the objects of $FM$ consists of lists of objects of $M$, including an empty list, which gives the identity object.
Given a source string $\br{x_i}_{i=1}^r=(x_1,\dots,x_r)$ and a target string $\br{y_j}_{j=1}^s=(y_1,\dots,y_s)$, we define
a morphism $\br{x_i}_{i=1}^r\to\br{y_j}_{j=1}^s$ to consist of a function $\phi:\{1,\dots,r\}\to\{1,\dots,s\}$ and, for each $j$ with $1\le j\le s$,
a morphism $\psi_j:\br{x_i}_{\phi(i)=j}\to y_j$ in $M$, where $\br{x_i}_{\phi(i)=j}$ is the tuple of entries in $\br{x_i}_{i=1}^r$ whose indices get mapped to $j$.  If there are no such $i$, then $\br{x_i}_{\phi(i)=j}$ is the empty list, and $\psi_j$ is a 0-morphism in $M$.  The permutative structure is
given by concatenation of lists.  
We remark that multifunctors are sent to \emph{strict} maps of permutative categories by this construction.

Exploiting Lemma \ref{unbased}, we begin by describing 
a bilinear map $\lambda:(FM,FN)\to F(M\otimes N)$ for multicategories $M$ and $N$, where $M\otimes N$
is the tensor product of multicategories originally due to Boardman and Vogt, and described in detail in \cite{EM2}.
In particular, the objects of $M\otimes N$ consist of $\Ob(M)\times\Ob(N)$,
with a typical object written $x\otimes y$,
 and the morphisms of $M\otimes N$
are generated by those of the form $x\otimes\psi$ or $\phi\otimes y$, where $x$ is an object of $M$, 
$\psi$ is a morphism of $N$, $\phi$ is a morphism of $M$, and
$y$ is an
object of $N$; these are the induced morphisms
from the universal bilinear map $(M,N)\to M\otimes N$ of multicategories.
See \cite{EM2}, Construction 4.10 and Proposition 4.16.

To give the bilinear map $\lambda:(FM,FN)\to F(M\otimes N)$ of permutative categories, we must first give
a map on objects $\Ob(FM)\times\Ob(FN)\to\Ob(F(M\otimes N))$.  We do so by assigning
\[
\lambda(\br{x_i}_{i=1}^r,\br{y_j}_{j=1}^s):=\br{\br{x_i\otimes y_j}_{i=1}^r}_{j=1}^s,
\]
with the indices prioritized as written.  Notice that if $r=0$ or $s=0$, then the result is the empty list in $F(M\otimes N)$,
which is the unit object, as required for a bilinear map of permutative categories.  

This assignment must give us a bifunctor on the underlying categories, so suppose given a morphism
$(f,\br{\psi_k}_{k=1}^t):\br{x_i}_{i=1}^s\to\br{z_k}_{k=1}^t$ in $FM$, so $f:[r]=\{1,\dots,r\}\to\{1,\dots,t\}=[t]$ and 
for each $k\in[t]$, we have $\psi_k:\br{x_i}_{f(i)=k}\to z_k$ in $M$, and similarly suppose given
$(g,\br{\xi_q}_{q=1}^p):\br{y_j}_{j=1}^s\to\br{w_q}_{q=1}^p$ in $FN$; we must produce an induced morphism
\[
\br{\br{x_i\otimes y_j}_{i=1}^r}_{j=1}^s\to\br{\br{z_k\otimes w_q}_{k=1}^t}_{q=1}^p.
\]
In order to do so, we first use the product map
\[
f\times g:[rs]\cong[r]\times[s]\to[t]\times[p]\cong[tp],
\]
where the bijections are given by lexicographic order.
We then observe that for $(k,q)\in[t]\times[p]$, we have $(f\times g)^{-1}(k,q)=f^{-1}(k)\times g^{-1}(q)$, so
\[
\br{x_i\otimes y_j}_{(f\times g)(i,j)=(k,q)}
=\br{\br{x_i\otimes y_j}_{f(i)=k}}_{g(j)=q}.
\]
The required map $\br{x_i\otimes y_j}_{(f\times g)(i,j)=(k,q)}\to z_k\otimes w_q$ is then given by
either way of traversing the bilinearity rectangle
\[
\xymatrix@C+10pt{
\br{\br{x_i\otimes y_j}_{f(i)=k}}_{g(j)=q}\ar@{<->}[d]_-{\cong}\ar[r]^-{\br{\phi_k\otimes y_j}}
&\br{z_k\otimes y_j}_{g(j)=q}\ar[dd]^-{z_k\otimes \xi_q}
\\\br{\br{x_i\otimes y_j}_{g(j)=q}}_{f(i)=k}\ar[d]^-{\br{x_i\otimes\xi_q}}
\\\br{x_i\otimes w_q}_{f(i)=k}\ar[r]_-{\phi_k\otimes w_q}
&z_k\otimes w_q.
}
\]
This gives us the required bifunctor underlying our bilinear map.

We also require distributivity maps 
\[
\delta_1:\lambda(\br{x_i}_{i=1}^r,\br{y_j}_{j=1}^s)\odot\lambda(\br{z_k}_{k=1}^t,\br{y_j}_{j=1}^s)
\to\lambda(\br{x_i}_{i=1}^s\odot\br{z_k}_{k=1}^t,\br{y_j}_{j=1}^s)
\]
and
\[
\delta_2:\lambda(\br{x_i}_{i=1}^r,\br{y_j}_{j=1}^s)\odot\lambda(\br{x_i}_{i=1}^r,\br{w_q}_{q=1}^p)
\to\lambda(\br{x_i}_{i=1}^r,\br{y_j}_{j=1}^s\odot\br{w_q}_{q=1}^p)
\]
subject to the coherence conditions of \cite{EM1}, Definition 3.2.  For $\delta_1$, we expand the source
and obtain
\begin{gather*}
\lambda(\br{x_i}_{i=1}^r,\br{y_j}_{j=1}^s)\odot\lambda(\br{z_k}_{k=1}^t,\br{y_j}_{j=1}^s)
\\=\br{\br{x_i\otimes y_j}_{i=1}^r}_{j=1}^s\odot\br{\br{z_k\otimes y_j}_{k=1}^t}_{j=1}^s,
\end{gather*}
while expanding the target gives us
\begin{gather*}
\lambda(\br{x_i}_{i=1}^s\odot\br{z_k}_{k=1}^t,\br{y_j}_{j=1}^s)
\\=\br{\br{x_i\otimes y_j}_{i=1}^r\odot\br{z_k\otimes y_j}_{k=1}^t}_{j=1}^s.
\end{gather*}
Shuffling from one side to the other gives us a well-defined element of $\Sigma_{s(r+t)}$, which we
adopt as our definition of $\delta_1$.

For $\delta_2$, expanding the source gives us
\begin{gather*}
\lambda(\br{x_i}_{i=1}^r,\br{y_j}_{j=1}^s)\odot\lambda(\br{x_i}_{i=1}^r,\br{w_q}_{q=1}^p)
\\=\br{\br{x_i\otimes y_j}_{i=1}^r}_{j=1}^s\odot\br{\br{x_i\otimes w_q}_{i=1}^r}_{q=1}^p,
\end{gather*}
while expanding the target gives us
\begin{gather*}
\lambda(\br{x_i}_{i=1}^r,\br{y_j}_{j=1}^s\odot\br{w_q}_{q=1}^p)
\\=\br{\br{x_i\otimes y_j}_{i=1}^r}_{j=1}^s\odot\br{\br{x_i\otimes w_q}_{i=1}^r}_{q=1}^p,
\end{gather*}
which is exactly the same thing.  We therefore use the identity for $\delta_2$, and the coherence relations
for just $\delta_2$ follow immediately.  The other coherence relations involving $\delta_1$ and both $\delta_1$
and $\delta_2$ follow from the fact that they all involve a well-defined shuffling of terms from one side to the other.
We have therefore constructed the desired bilinear $\lambda:(FM,FN)\to F(M\otimes N)$.

\ignore
Applying this construction to the universal bilinear map $(M,N)\to M\otimes N$ in $\Mult$, we recall 
from \cite{EM2}, Corollary 4.16, 
that
$M\otimes N$ has objects $(\Ob M)\times(\Ob N)$,
with objects generically denoted $m\otimes n$, 
and morphisms generated by those of the form
$\phi\otimes n$ and $m\otimes\psi$, where $\phi$ is an arbitrary morphism in $M$ and $\psi$ is an
arbitrary morphism in $N$.  Our desired bilinear map $\lambda:(FM,FN)\to F(M\otimes N)$ then sends
$(\br{x_i}_{i=1}^r,\br{y_j}_{j=1}^s)$ to $\br{\br{x_i\otimes y_j}_{i=1}^r}_{j=1}^s$, with the construction on
morphisms following the construction given above.  
\endignore
%{\bf{IT MIGHT BE BETTER TO JUST GIVE THIS AS THE CONSTRUCTION TO BEGIN WITH.}}

Now the objects of both $((M\otimes N)\otimes P)$ and $(M\otimes(N\otimes P))$ take the form
of $x\otimes y\otimes z$, with the only difference being the insertion of parentheses, and similarly
with the generating morphisms, which are of the form $\phi\otimes y\otimes z$, $x\otimes\psi\otimes z$,
or $x\otimes y\otimes\xi$.  The associativity coherence diagram for $\lambda$ now follows by inspection.
So do the consistency with transposition, and the naturality diagram.  
\ignore
The other two diagrams have
to do with basepoint preservation, which is not present in this context, so we don't have a based
multifunctor, nor do we need one.  
\endignore
Theorem \ref{step3} therefore follows, and this concludes the
proof that Mandell's construction is multiplicative.

\end{document}